\documentclass[12pt]{amsart}
\usepackage{amssymb,verbatim,amscd,amsmath,graphicx,enumerate}
\usepackage{graphicx}
\usepackage{caption}
\usepackage{subcaption}
\usepackage{pdfsync}
\usepackage{fullpage}
\usepackage{color}
\usepackage{marginnote}
\usepackage{tikz}

\usetikzlibrary{patterns}
\usetikzlibrary{calc}
\usetikzlibrary{matrix}
\usepgflibrary{arrows}
\usetikzlibrary{arrows}
\usetikzlibrary{decorations.pathreplacing}
\usetikzlibrary{decorations.pathmorphing}
\usepackage{hyperref}

\numberwithin{equation}{section}
\newtheorem{theorem}{Theorem}[section]

\newtheorem{proposition}[theorem]{Proposition}
\newtheorem{corollary}[theorem]{Corollary}

\theoremstyle{definition}
\newtheorem{definition}[theorem]{Definition}

\newtheorem{def-prop}[theorem]{Definition-Proposition}
\newtheorem{remark}[theorem]{Remark}
\newtheorem{example}[theorem]{Example}

\newtheorem*{acknowledgement}{Acknowledgements}

\newtheorem*{Mysketch}{Sketch of proof} 
  {\pushQED{\qed}\begin{Mysketch}}
  {\popQED\end{Mysketch}}

\DeclareMathOperator{\depth}{depth}

\def\NN{\mathbb{N}}
\def\ZZ{\mathbb{Z}}

\begin{document}

\title{Depth of powers of squarefree monomial ideals}

\author{Louiza Fouli}
\address{Department of Mathematical Sciences \\
New Mexico State University\\
P.O. Box 30001 \\
Department 3MB \\
Las Cruces, NM 88003}
\email{lfouli@nmsu.edu}
\urladdr{http://www.web.nmsu.edu/~lfouli}

\author{Huy T\`ai H\`a}
\address{Department of Mathematics \\
Tulane University \\
6823 St. Charles Avenue \\
New Orleans, LA 70118}
\email{tha@tulane.edu}
\urladdr{http://www.math.tulane.edu/~tai/}

\author{Susan Morey}
\address{Department of Mathematics \\
Texas State University\\
601 University Drive\\
San Marcos, TX 78666}
\email{morey@txstate.edu}
\urladdr{http://www.txstate.edu/~sm26/}


\begin{abstract}
We derive two general bounds for the depths of powers of squarefree monomial ideals corresponding to hyperforests. These bounds generalize known bounds for the depths of squarefree monomial ideals, which were given in terms of the edgewise domination number of the corresponding hypergraphs and the lengths of initially regular sequences with respect to the ideals.
\end{abstract}

\maketitle

\section{Introduction}\label{intro}

During the past two decades, many papers have appeared with various approaches to computing lower bounds for the depth, or equivalently upper bounds for the projective dimension, of $R/I$ for a \emph{squarefree} monomial ideal $I$ (cf. \cite{DS, DS2, HHKO1, LM1, LM2, P3}). The general idea has been to associate to the ideal $I$ a graph or hypergraph $G$ and use \emph{dominating} or \emph{packing} invariants of $G$ to bound the depth of $R/I$.

In general, given an ideal $I \subseteq R$, it is not just the depth of $R/I$ that attracts significant attention; rather, it is the entire \emph{depth function} $\depth R/I^s$, for $s \in \NN$.
A result by Burch, that was later improved by Brodmann, states that $\lim \limits_{s \rightarrow \infty} \depth R/I^{s} \leq \dim R -\ell(I)$, where $\ell(I)$ is the analytic spread of $I$ \cite{Brod, Bur}. Moreover, Eisenbud and Huneke \cite{EH} showed that if, in addition, the associated graded ring, ${\rm{gr}}_{I}(R)$, of $I$ is Cohen-Macaulay, then the above inequality becomes an equality. Therefore, one can say that the limiting behavior of the $\depth R/I^s$ is quite well understood. 
It is then natural to consider the initial behavior of the depth function (cf. \cite{BHH, FM, HS, HangT, HH2, HV, HKTT, HoaT, MST, MN, Morey, MV, TT}).

Examples have been exhibited to show that the initial behavior of $\depth R/I^s$ can be wild, see \cite{BHH}. In fact, it was conjectured by Herzog and Hibi \cite{HH2} that for any numerical function $f: \NN \rightarrow \ZZ_{\geq 0}$ that is asymptotically constant, there exists an ideal $I \subseteq R$ in a polynomial ring such that $f(s) = \depth R/I^s$ for all $s \geq 1$. This conjecture has recently been resolved affirmatively in \cite{HNTT}. It was proven in \cite{HNTT} that the depth function of a monomial ideal can be any numerical function that is asymptotically constant. Yet, it is still not clear what depth functions are possible for squarefree monomial ideals.

Unlike the case for $\depth R/I$, few lower bounds for $\depth R/I^s$, $s \in \NN$, are known (cf. \cite{FM, Morey, TT}). One reason for this is that powers of squarefree monomial ideals are not squarefree and so many of the known bounds for $R/I$ do  not apply to $R/I^s$. To address this situation, we adapt a proof technique from \cite{BHT} to  generalize bounds for $\depth R/I$, that were given by Dao and Schweig \cite{DS2}, in terms of the \emph{edgewise domination number}, and by the authors \cite{FHM}, in terms of the length of an \emph{initially regular sequence}. We provide lower bounds for the depth function $\depth R/I^s$, $s \in \NN$, when $I$ is a squarefree monomial ideal corresponding to a hyperforest or a forest, Theorems \ref{thm.powerforest} and \ref{star.bound.powers}. 

Our results, Theorems \ref{thm.powerforest} and \ref{star.bound.powers}, predict correctly the general behavior, as computation indicates for random hyperforests and forests, that the depth function $\depth R/I(G)^s$ decreases incrementally as $s$ increases. For specific examples, our bound in Theorem \ref{thm.powerforest} could be far from the actual values of the depth function -- and this is because the starting bound for $\depth R/I$ in terms of the edgewise domination number is not always optimal. For forests, Theorem \ref{star.bound.powers} could provide a more accurate starting bound for $\depth R/I$ using initially regular sequences and, thus, be closer to the depth function.

The common important underlying idea behind Theorems \ref{thm.powerforest} and \ref{star.bound.powers} is that if $\alpha(G)$ is an invariant associated to a hyperforest $G$ that gives the initial bound $\depth R/I(G) \ge \alpha(G)$ and satisfies a certain inequality when restricted to subhypergraphs then one should have
$$\depth R/I(G)^s \ge \max\{\alpha(G)-s+1, 1\}.$$
Our work in this paper, thus, could be interpreted as the starting point of a research program in finding such combinatorial invariants $\alpha(G)$ to best describe the depth function of squarefree monomial ideals, which we hope to continue to pursue in future works.

\begin{acknowledgement} The first author was partially supported by a grant from the Simons Foundation (grant \#244930). The second named author is partially supported by Louisiana Board of Regents (grant \#LEQSF(2017-19)-ENH-TR-25). We also thank Seyed Amin Seyed Fakhari for pointing out an error in a previous version of the article. 
\end{acknowledgement}

\section{Background} \label{sec.ini}

For unexplained terminology, we refer the reader to \cite{BH} and \cite{HH}. Throughout the paper, $R = k[x_1, \ldots, x_n]$ is a polynomial ring over an arbitrary field $k$ and all hypergraphs will be assumed to be {\em simple}, that is, there are no containments among the edges. For a hypergraph $G = (V_G,E_G)$ over the vertex set $V_G = \{x_1, \ldots, x_n\}$, the \emph{edge ideal} of $G$ is defined to be
$$I(G) = \left\langle \prod_{x \in e} x ~\Big|~ e \in E_G\right\rangle \subseteq R.$$
This construction gives a one-to-one correspondence between squarefree monomial ideals in $R = k[x_1, \ldots, x_n]$ and (simple) hypergraphs on the vertex set $V = \{x_1, \ldots, x_n\}$. 

For a vertex $x$ in a graph or hypergraph $G$, we say $y$ is a {\em neighbor} of $x$ if there exists an edge $E \in E_G$ such that $x,y \in E$. The {\em neighborhood} of $x$ in $G$ is $N_G(x) = \{ y \in V_G \mid y \, {\mbox{\rm is a neighbor of}} \, x\}$. The {\em closed neighborhood} of $x$ in $G$ is $N_G[x]= N_G(x) \cup \{x\}$. Note that the $G$ in the notation will be suppressed when it is clear from context. 

Simplicial forests were defined by Faridi in \cite{Faridi}, where it was shown that the edge ideals of these hypergraphs are always sequentially Cohen-Macaulay. They have also been used in the study of standard graded (symbolic) Rees algebras of squarefree monomial ideals \cite{HHTZ}. We first recall the definition of a simplicial forest (or a hyperforest for short).
\begin{definition} Let $G = (V,E)$ be a simple hypergraph.
\begin{enumerate}
\item An edge $e \in E$ is called a \emph{leaf} if either $e$ is the only edge in $G$ or there exists $e \not= g \in E$ such that for any $e \not= h \in E$, $e \cap h \subseteq e \cap g$.
\item A leaf $e$ in $G$ is called a \emph{good} leaf if the set $\{e \cap h \mid h \in E\}$ is totally ordered with respect to inclusion.
\item $G$ is called a \emph{simplicial forest} (or simply, a \emph{hyperforest}) if every subhypergraph of $G$ contains a leaf. A \emph{simplicial tree} (or simply, a \emph{hypertree}) is a connected hyperforest.
\end{enumerate}
\end{definition}

It follows from \cite[Corollary 3.4]{HHTZ} that every hyperforest contains good leaves. It is also immediate that every graph that is a forest is also a hyperforest.

\begin{example} For the hypergraphs depicted below, the first one is not a hypertree while the second one is, see also \cite[Examples~1.4, 3.6]{Faridi}.

\setlength{\unitlength}{0.7cm}
\begin{tikzpicture}
    \tikzstyle{point}=[circle,thick,draw=black,fill=black,inner sep=0pt,minimum width=4pt,minimum height=4pt]

 \node (a)[point,label={[xshift=-0.3cm, yshift=0 cm]: $\bf{a}$}] at (0,0) {};

    \node (b)[point,label={[label distance=0cm]0: $\bf{b}$}] at (3.5,0) {};

\node (c)[point,label={[label distance=0cm]0: $\bf{c}$}] at (1.5,3) {};

    \node (d)[point,label={[label distance=0cm]0:$\bf{d}$}] at (1.5,1) {};

 \draw[pattern=north east lines] (a.center) -- (b.center) -- (d.center) -- cycle;
    \draw[pattern=north west lines] (a.center) -- (b.center) -- (c.center) -- cycle;
    \draw[pattern=vertical lines]   (c.center) -- (b.center) -- (d.center) -- cycle;

  \node (x)[point,label={[xshift=-0.7cm, yshift=0 cm]0:$\bf{x}$}] at (5,0) {};

    \node (y)[point,label={[xshift=-0.7cm, yshift=0 cm]0:$\bf{y}$}] at (5,2) {};

\node (z)[point,label={[label distance=0cm]0:$\bf{z}$}] at (7,0) {};

    \node (u)[point,label={[label distance=0cm]0:$\bf{u}$}] at (7,2) {};

       \node (w)[point,label={[label distance=0cm]0:$\bf{v}$}] at (9,0) {};

 \draw[pattern=north east lines] (x.center) -- (y.center) -- (z.center) -- cycle;
    \draw[pattern=north west lines] (y.center) -- (z.center) -- (u.center) -- cycle;
    \draw   (u.center) -- (w.center);

\end{tikzpicture}
\end{example}

In this paper, we will focus on two invariants that are known to bound the depth of $R/I$ when $I$ is the edge ideal of an arbitrary hypergraph. When $G$ is a simplicial forest, we will provide a linearly decreasing lower bound for the depths of the powers of $I$ using each of these invariants. The first of these bounds for the depth function of a squarefree monomial ideal is the edgewise domination number introduced in \cite{DS2}. Recall that for a hypergraph $G = (V,E)$, a subset $F \subseteq E$ is called \emph{edgewise dominant} if for every vertex $v \in V$ either $\{v\} \in E$ or $v$ is adjacent to a vertex contained in an edge of $F$. 

\begin{definition}[\protect{\cite{DS2}}]
The \emph{edgewise domination number} of $G$ is defined to be
$$\epsilon(G) = \min \{|F| \mid F \subseteq E \text{ is edgewise dominant}\}.$$
\end{definition}

The second invariant used in this paper will be a variation on the depth bound for monomial ideals introduced in \cite{FHM}. For an arbitrary vertex $b_0$ in a hypergraph $G$, define a {\em star} on $b_0$ to be a linear sum $b_0+b_1+\cdots + b_t$ such that for each edge $E_i$ of $G$, if $b_0 \in E_i$, then there exists a $j > 0$ such that $b_j \in E_i$. It was shown in \cite[Theorem 3.11]{FHM} that a set of vertex-disjoint stars that can be embedded in a hypergraph $G$ forms an initially regular sequence and, thus, gives a lower bound for the depth of $R/I(G)$. While much of \cite{FHM} focuses on strengthening this bound by weakening the disjoint requirement and allowing for additional types of linear sums, in this article we will apply the bound to graphs, where the situation is more restricted. Notice that for a graph $G$, a star on $b_0$ is the sum of all vertices in the closed neighborhood of $b_0$, while for a hypergraph, a subset of the closed neighborhood can suffice. A {\em star packing} is a collection $S$ of vertex-disjoint stars in $G$ such that if $x \in V_G$ then $N_G[x] \cap {\mbox{\rm Supp}}(S) \not= \emptyset$. In other words, $S$ is maximal in the sense that no additional disjoint stars exist. This leads to the following definition, whose notation reflects its relationship to a $2$-packing of closed neighborhoods in graph theory. 
\begin{definition}\label{star_packing}
The {\em star packing number} $\alpha_2$ of a hypergraph $G$ is given by
$$\alpha_2(G) = \max \{ |S| \mid S \, {\mbox{\rm is a star packing of}}\, G\}.$$
\end{definition}

\begin{remark}\label{unused.vertices}
If $x_1, \ldots, x_k \in R$ are variables in $R$ that do not appear in any edge of $G$, then $x_1 , \ldots , x_k$ is a regular sequence on $R/I(G)$ and $\depth R/I(G) = k + \depth R/(x_1, \ldots , x_k, I(G))$. 
\end{remark}

Note that if $S$ is any set of disjoint stars in a hypergraph $G$, then $\alpha_2(G) \geq |S|$ since $S$ can be extended to a full star packing. Note also that for the special case when $G$ is a graph, a star packing is equivalent to a closed neighborhood packing and, by focusing on the centers of the stars, to a maximal set of vertices such that the distance between any two is at least $3$.

\section{Depth of powers of squarefree monomial ideals}\label{epsilon}

In this section, we use a technique introduced in \cite{BHT} to give a general lower bound for the depth function of a squarefree monomial ideal when the underlying hypergraph is a hyperforest (also known as a simplicial forest). 
In the case of a forest, we extend the result to show that an additional, often stronger, bound holds. For simplicity of notation, we write $V_G$ and $E_G$ to denote the vertex and edge sets of a hypergraph $G$.

\begin{theorem} \label{thm.powerforest}
Let $G$ be a hyperforest with at least one edge of cardinality at least $2$, and let $I = I(G)$. Then for all $s \geq 1$,
$$\depth R/I^s \geq \max\{\epsilon(G)-s+1, 1\}.$$
\end{theorem}

\begin{proof} It follows from \cite[Corollary 3.3]{HHTZ} (see also \cite{GRV}) that the symbolic Rees algebra of $I$ is standard graded. That is, $I^{(s)} = I^s$ for all $s \ge 1$. In particular, this implies that $I^s$ has no embedded primes for all $s \ge 1$. Thus, $\depth R/I^s \ge 1$ for all $s \ge 1$.

It remains to show that $\depth R/I^s \ge \epsilon(G) - s + 1$. Indeed, this statement and, hence, Theorem~\ref{thm.powerforest} follows from the following slightly more general result.
\end{proof}

\begin{proposition} \label{prop.powerforest}
Let $G$ be a hyperforest. Let $H$ and $T$ be subhypergraphs of $G$ such that
$$E_H \cup E_T = E_G \text{ and } E_H \cap E_T = \emptyset.$$
Then we have
$$\depth R/[I(H)+I(T)^s] \geq \max\{\epsilon(G)-s+1, 0\}.$$
\end{proposition}

\begin{proof}\phantom{\qedhere}
It suffices to show that $\depth R/[I(H)+I(T)^s] \geq \epsilon(G)-s+1.$ We shall use induction on $|E_T|$ and $s$. If $|E_T| = 0$ then the statement follows from \cite[Theorem 3.2]{DS}. If $s = 1$ then the statement also follows from \cite[Theorem 3.2]{DS}. Suppose that $|E_T| \geq 1$ and $s \geq 2$.

Let $e$ be a good leaf of $T$. By the proof of \cite[Theorem 5.1]{CHHKTT}, we have
$I(T)^s : e = I(T)^{s-1}.$
This implies that
$$(I(H)+I(T)^s):e = (I(H):e) + I(T)^{s-1}.$$
Moreover,
$$I(H)+I(T)^s+(e) = I(H+e) + I(T\setminus e)^s.$$
Thus, we have the exact sequence
$$0 \rightarrow R/[(I(H):e)+I(T)^{s-1}] \rightarrow R/[I(H)+I(T)^s] \rightarrow R/[I(H+e)+I(T\setminus e)^s] \rightarrow 0$$
which, in turns, gives
\begin{equation}
\depth R/[I(H)+I(T)^s] \geq
 \min\{\depth R/[(I(H):e)+I(T)^{s-1}], \depth R/[I(H+e)+I(T \setminus e)^s]\}.\label{eq.111}
\end{equation}

Observe that $G = (H+e) + (T \setminus e)$ and $E_{H+e} \cap E_{T \setminus e} = \emptyset$. Thus, by induction on $|E_T|$, we have
$$\depth R/[I(H+e) + I(T \setminus e)^s] \geq \epsilon(G)-s+1.$$
On the other hand, let $Z = \{ z \in V_H \mid \exists h \in E_H \text{ such that } \{z\} = h \setminus e\}$. Let $H'$ be the hypergraph obtained from $I(H):e$ by deleting the vertices in $Z$ and any vertex in $H$ that does not belong to any edge.  Let $T'$ be the hypergraph whose edges are obtained from edges of $T$ after deleting all those that contain any vertex in $V_T\cap Z$. Then
$$I(H):e = I(H') + (z \mid z \in Z).$$

Let $G' = H' + T'$, let $R' = k[V_{H'} \cup V_{T'}]$, and let $W = V_G \setminus (V_{G'} \cup Z)$. It follows by induction on $s$ that
\begin{align*}
\depth R/[(I(H):e)+I(T)^{s-1}] & = \depth R/[I(H')+I(T')^{s-1}+(z \mid z \in Z)] \\
& = \depth R'/[I(H')+I(T')^{s-1}] + |W| \\
& \geq \epsilon(G')-(s-1)+1 + |W| \\
&= \left(\epsilon(G')+1+|W|\right)-s+1.
\end{align*}

Now, let $F' \subseteq E_{G'}$ be an edgewise dominant set in $G'$. By the construction of $H'$, for each $f' \in F' \cap E_{H'}$, there is an edge $f \in E_H$ such that $f' = f \setminus e$. Let $F$ be the set obtained from $F'$ by replacing each $f' \in F' \cap E_{H'}$ by such $f$. Observe that for any vertex $v \in V_G$, either $v \in W$, or $v \in Z$, or $v \in V_{G'}$. If $v \in Z$ then $v$ is dominated by $e$. If $v \in V_{G'}$ then $v$ is dominated by some edge in $F'$. Thus, $F \cup \{e\}$ together with one edge for each vertex in $W$ will form an edgewise dominant set in $G$. This implies that
$$\epsilon(G')+1+|W| \geq \epsilon(G).$$
Therefore,
$$\depth R/[(I(H):e) + I(T)^{s-1}] \geq \epsilon(G)-s+1.$$
Hence, by (\ref{eq.111}), we have
\[
\pushQED{\qed}
\depth R/[I(H)+I(T)^s] \geq \epsilon(G)-s+1. \qedhere
\popQED
\]
\end{proof}

A close examination of the proof of Proposition~\ref{prop.powerforest} shows that we can replace $\epsilon(G)$ by any invariant $\alpha(G)$, for which $\depth R/I(G) \ge \alpha(G)$ and
$\alpha(G') + 1 + |W| \geq \alpha(G),$
where $G'$ and $W$ are defined as in the proof of Proposition~\ref{prop.powerforest}.

\begin{corollary}\label{key.point} 
If $\alpha(G)$ is any invariant of a hyperforest $G$ for which $\depth R/I(G) \ge \alpha(G)$ and
$\alpha(G') + 1 + |W| \geq \alpha(G),$
then
$$\depth R/I^s \geq \max\{\alpha(G)-s+1, 0\}.$$
\end{corollary}

For a random hypertree $G$, computations indicate that the depth function $\depth R/I(G)^s$ decreases incrementally as $s$ increases as  predicted by Theorem~\ref{thm.powerforest}. However, for low powers of $I$, the $\epsilon$-bound is often less than optimal, as can be seen by comparing the results to the bounds on $\depth R/I(G)$ obtained from \cite{FHM}. For hypertrees $G$ for which $\depth R/I(G) = \epsilon$, the depth function $\depth R/I(G)^s$ usually does not initially decrease incrementally as $s$ increases. These statements are illustrated by the following pair of examples.

\begin{example} \label{ex.111}
Let $I=(x_1x_2, x_2x_3, x_3x_4, x_3x_5, x_3x_6, x_6x_7,x_6x_8, x_8x_9, x_8x_{10}, x_8x_{11}, x_8x_{12})\subseteq R=\mathbb{Q}[x_1, \ldots, x_{12}]$ be the edge ideal of the graph $G$ depicted below.

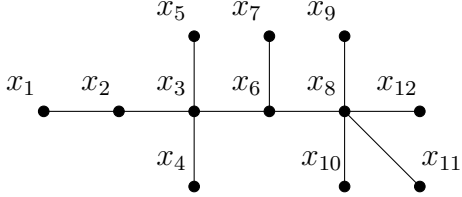
\begin{tikzpicture}

 \tikzstyle{point}=[circle,draw=black,fill=black,inner sep=0pt,minimum width=4pt,minimum height=4pt]

 \node (a)[point,label={[xshift=-0.3cm, yshift=0 cm]: ${x_1}$}] at (0,0) {};

    \node (b)[point,label={[xshift=-0.3cm, yshift=0 cm]: ${x_2}$}] at (1,0) {};

\node (c)[point,label={[xshift=-0.3cm, yshift=0 cm]: ${x_3}$}] at (2,0) {};

\node (d)[point,label={[xshift=-0.3cm, yshift=0 cm]: ${x_4}$}] at (2,-1) {};

\node (e)[point,label={[xshift=-0.3cm, yshift=0 cm]:${x_5}$}] at (2,1) {};

\node (f)[point,label={[xshift=-0.3cm, yshift=0 cm]: ${x_6}$}] at (3,0) {};

\node (g)[point,label={[xshift=-0.3cm, yshift=0 cm]: ${x_7}$}] at (3,1) {};

\node (h)[point,label={[xshift=-0.3cm, yshift=0 cm]: ${x_8}$}] at (4,0) {};

 \node (i)[point,label={[xshift=-0.3cm, yshift=0 cm]: ${x_9}$}] at (4,1) {};

\node (j)[point,label={[xshift=-0.3cm, yshift=0 cm]: ${x_{10}}$}] at (4,-1) {};

 \node (k)[point,label={[xshift=0.3cm, yshift=0 cm]: ${x_{11}}$}] at (5,-1) {};

\node (n)[point,label={[xshift=-0.3cm, yshift=0 cm]: ${x_{12}}$}] at (5,0) {};

\draw (a.center) -- (b.center) -- (c.center) -- (f.center) -- (h.center) -- (n.center);

 \draw (c.center) -- (e.center);

 \draw (c.center) -- (d.center);
 \draw (f.center) -- (g.center);
 \draw (h.center) -- (i.center);
  \draw (h.center) -- (j.center);
   \draw (h.center) -- (k.center);
\end{tikzpicture}

\noindent Computation in Macaulay~2 \cite{M2} shows that the depth function of $I$ is $4,3,2,1,1, \ldots$. Thus, Theorem~\ref{thm.powerforest} predicts correctly how the depth function behaves. However, in this example, $\epsilon(G) = 2$ does not give the right value for $\depth R/I$. 
\end{example}

\begin{example} \label{ex.222}
Let $I=(x_1x_2, x_1x_3, x_1x_4, x_4x_5, x_5x_6, x_5x_7,x_4x_8, x_8x_9, x_8x_{10}, x_8x_{11}, x_8x_{12})\subseteq R=\mathbb{Q}[x_1, \ldots, x_{12}]$ be the edge ideal of the graph $G$ depicted below.

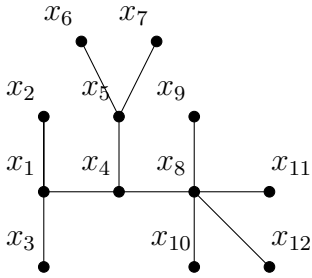
\begin{tikzpicture}

 \tikzstyle{point}=[circle,draw=black,fill=black,inner sep=0pt,minimum width=4pt,minimum height=4pt]

 \node (a)[point,label={[xshift=-0.3cm, yshift=0 cm]: ${x_1}$}] at (1,0) {};

    \node (b)[point,label={[xshift=-0.3cm, yshift=0 cm]: ${x_2}$}] at (1,1) {};

\node (c)[point,label={[xshift=-0.3cm, yshift=0 cm]: ${x_3}$}] at (1,-1) {};

\node (d)[point,label={[xshift=-0.3cm, yshift=0 cm]: ${x_4}$}] at (2,0) {};

\node (e)[point,label={[xshift=-0.3cm, yshift=0 cm]:${x_5}$}] at (2,1) {};

\node (f)[point,label={[xshift=-0.3cm, yshift=0 cm]: ${x_6}$}] at (1.5,2) {};

\node (g)[point,label={[xshift=-0.3cm, yshift=0 cm]: ${x_7}$}] at (2.5,2) {};

\node (h)[point,label={[xshift=-0.3cm, yshift=0 cm]: ${x_8}$}] at (3,0) {};

 \node (i)[point,label={[xshift=-0.3cm, yshift=0 cm]: ${x_9}$}] at (3,1) {};

\node (j)[point,label={[xshift=-0.3cm, yshift=0 cm]: ${x_{10}}$}] at (3,-1) {};

 \node (k)[point,label={[xshift=0.3cm, yshift=0 cm]: ${x_{11}}$}] at (4,0) {};

\node (n)[point,label={[xshift=0.3cm, yshift=0 cm]: ${x_{12}}$}] at (4,-1) {};

\draw (a.center) -- (b.center) -- (c.center) ;
\draw (a.center) -- (d.center) -- (e.center) -- (f.center);

 \draw (e.center) -- (g.center);

 \draw (d.center) -- (h.center);
 \draw (h.center) -- (i.center);
  \draw (h.center) -- (j.center);
   \draw (h.center) -- (k.center);
   \draw (h.center) -- (n.center);
\end{tikzpicture}

Then $\epsilon(G) = 3$. Computation in Macaulay~2 \cite{M2} shows that the depth function of $I$ is $3,3,3,1,1,\ldots$. The bound in Theorem~\ref{thm.powerforest} gives the depth function of $I$ to be at least $3, 2, 1, 1, 1, \ldots$. In this example, while $\epsilon(G)$ gives the right value for $\depth R/I$, Theorem~\ref{thm.powerforest} does not predict correctly how the depth function of $I$ behaves.
\end{example}

Examples~\ref{ex.111} and \ref{ex.222} show that to get a sharp bound for the depth function of random hypertrees, we may want to start with invariants other than $\epsilon(G)$ which give stronger bounds for $\depth R/I(G)$. In order to do so, one often needs to assume additional structure on $G$. For example, if $G$ is a forest, the invariant from Definition~\ref{star_packing} can be used.

\begin{proposition}\label{prop.graph.powers}
Let $G$ be a forest with connected components $G_1, \ldots , G_t$.  Let $H$ and $T$ be subforests of $G$ such that $E_H \cup E_T = E_G$, $ E_H \cap E_T = \emptyset$, and $T \cap G_i$ is connected for each $i$. Then 
$$\depth R/[I(H)+I(T)^s] \geq \max\{\alpha_2(G)-s+1, 0\}.$$
\end{proposition}

\begin{proof}
The proof follows the outline of that of Proposition~\ref{prop.powerforest} with special care toward the end. If $|E_T| = 0$  or $s=1$, then the statement follows from \cite[Theorem 3.11]{FHM}, so we assume $|E_T| \geq 1$ and $s \geq 2$.

Consider an edge $\{ x, y\}$ of $H$. Then, $\{ x, y \} \in G_i$ for some $i$. Since $T \cap G_i$ is connected, if $x,y \in V_T$, then there is a path in $T$ from $x$ to $y$. This path, together with $\{x,y\}$, forms a cycle in $G$, which is a contradiction. Thus, no edge of $H$ can have both endpoints in $V_T$. 

Let $e$ be a leaf of $T$. Since $T$ is a forest, $e$ is a good leaf of $T$. Thus, as in the proof of Proposition~\ref{prop.powerforest}, we have
\begin{equation}
\depth R/[I(H)+I(T)^s] \geq
 \min\{\depth R/[(I(H):e)+I(T)^{s-1}], \depth R/[I(H+e)+I(T \setminus e)^s]\}.\label{eq.222}
\end{equation}
Observe further that $G = (H+e) + (T \setminus e)$, $E_{H+e} \cap E_{T \setminus e} = \emptyset$, and  $(T\setminus e) \cap G_i$ is connected for each $i$. Thus, by induction on $|E_T|$, we have
$$\depth R/[I(H+e) + I(T \setminus e)^s] \geq \alpha_2(G)-s+1.$$

On the other hand, let $Z = \{ z \in V_H \mid \exists h \in E_H \text{ such that } \{z\} = h \setminus e\}$. Let $H'$ be the graph obtained from $I(H):e$ by deleting the vertices in $Z$ and any vertex of $H$ that does not belong to any edge.  Note that $V_T\cap Z= \emptyset$, since otherwise there would be an edge of $H$ having both endpoints in $V_T$ (one in $Z$ and the other in $e$). Then
$$I(H):e = I(H') + (z \mid z \in Z).$$

 Let $G' = H' + T$, let $R' = k [V_{H'} \cup V_T]$, and let $W = V_G \setminus (V_{G'} \cup Z)$. 
 It follows by induction on $s$ that
\begin{align*}
\depth R/[(I(H):e)+I(T)^{s-1}] & = \depth R/[I(H')+I(T)^{s-1}+(z \mid z \in Z)] \\
& = \depth R'/[I(H')+I(T)^{s-1}] + |W| \\
& \geq \alpha_2(G')-(s-1)+1 + |W| \\
&= \left(\alpha_2(G')+1+|W|\right)-s+1.
\end{align*}

We will show that $\alpha_2(G') + 1 + |W| \geq \alpha_2(G)$. Fix a set of disjoint stars of $G$ of cardinality $\alpha_2(G)$ and let $S=\{x_1, \ldots , x_{\alpha_2(G)}\}$ denote the set of the centers of these stars. 

Let $S'=\{x_i \mid x_i \in R'\}$ and notice that the set of stars in $G'$ centered at $x_i$ for each $x_i \in R'$ is a set of disjoint stars and thus $\alpha_2(G') \geq |S'|$. 
If $x_i \not\in R'$, then $x_i \in Z \cup W$. Since the stars with centers in $S$ are disjoint, there can be at most two elements in $ Z\cap S$.  If $|Z \cap S| \leq 1$, then $|S'|\geq |S|-1-|W|=\alpha_2(G)-1-|W|$, and so $\alpha_2(G') + 1 +|W| \ge \alpha_2(G)$. 

Suppose that $|Z \cap S| = 2$. Write $e=ab$ and notice that if either $a$ or $b$ is in $S$, then $Z\cap S =\emptyset$. Hence, we may assume that $a, b \not\in S$. We will construct a new set of stars in $G'$ of cardinality at least $\alpha_2(G)-1-|W|$ and, thus, also give $\alpha_2(G') +1 + |W| \ge \alpha_2(G)$ in this case. 

Indeed, let  $\{z_1, z_2\} = Z\cap S$. Then, $z_1, z_2 \in N_G(a) \cup N_G(b)$ and, without loss of generality, we may assume that $z_1 \in N_G(a)$ and $z_2 \in N_G(b)$. Since $e$ is a leaf in $T$, we may also assume that $b$ is a leaf vertex in $T$; that is, $N_T(b)=a$. Then, $N_G(b) \setminus \{a\} \subseteq Z$.  Let $\widehat{S'}=S'\cup \{b\}$.  We claim that the stars in $G'$ centered on the elements of $\widehat{S'}$ are disjoint. Any two stars centered at elements of $S'$ are already disjoint. Consider then a star centered at an element $x_i \in S'$ and the star centered at $b$ in $G'$. Since $x_i \not= z_1$, and the stars in $G$ centered at $x_i$ and $z_1$ are disjoint, we have $a \not\in N_{G'}(x_i)$. Thus, $N_{G'}[x_i] \cap N_{G'}[b] = \emptyset$. Clearly, $|\widehat{S'}| \geq |S|-1-|W|=\alpha_2(G)-1-|W|$. 

Now, we have
$$\depth R/[(I(H):e) + I(T)^{s-1}] \geq \alpha_2(G)-s+1,$$
and the assertion now follows from (\ref{eq.222}).
\end{proof}

Using this result, we obtain the following bound which, while generally is stronger than that of Theorem~\ref{thm.powerforest} when applicable, applies only to graphs that are trees or forests. 

\begin{theorem}\label{star.bound.powers}
Let $G$ be a forest with at least one nontrivial edge, and let $I = I(G)$. Then,
$$\depth R/I^s \geq \max\{\alpha_2(G)-s+1, 1\}.$$
\end{theorem}

\begin{proof} It follows from \cite[Theorem 5.9]{SVV} that $I^{(s)} = I^s$ for all $s \ge 1$ and so $\depth R/I^s \ge 1$ for all $s \ge 1$.
By Proposition~\ref{prop.graph.powers}, $\depth R/I^s \ge \alpha_2(G) - s + 1$ and the result follows.
\end{proof}

\begin{example}
	Let $G$ be the graph in Example \ref{ex.111}. Using $x_1, x_5, x_7, x_9$ as centers of stars, we have $\alpha_2(G) = 4$. Thus, Theorem \ref{star.bound.powers} gives the correct depth function $\depth R/I(G)^s$, for all $s \in \NN$, for this graph.
	
	On the other hand, let $G$ be the graph as in Example \ref{ex.222}. Then, $\alpha_2(G) = 3 = \epsilon(G)$, and so Theorem \ref{star.bound.powers} gives the same bound as that of Theorem \ref{thm.powerforest} for this graph.
\end{example}

It would be interesting to know whether the length of a more general initially regular sequence with respect to $I(G)$, or improved bounds for $\depth R/I(G)$ obtained in \cite[Section 4]{FHM}, could be used to get better bounds for the depth function than those given in Theorem~\ref{thm.powerforest} when $G$ is a hyperforest.

\end{document}